\documentclass[12pt]{article}

\usepackage{ucs}
\usepackage{amssymb}
\usepackage{amsthm}
\usepackage{amsmath}
\usepackage{latexsym}
\usepackage[cp1251]{inputenc}
\usepackage[english]{babel}
\usepackage{graphicx}
\usepackage{wrapfig}
\usepackage{caption}
\usepackage{subcaption}
\usepackage{txfonts}
\usepackage{lmodern}
\usepackage{mathrsfs}
\usepackage{algorithm}
\usepackage{multicol}
\usepackage{enumerate}
\usepackage{titlesec}
\usepackage{paralist}
\usepackage{mathtools}

\titleformat{\paragraph}
{\normalfont\normalsize\bfseries}{\theparagraph}{1em}{}
\titlespacing*{\paragraph}
{0pt}{3.25ex plus 1ex minus .2ex}{1.5ex plus .2ex}

\newcommand*{\myproofname}{Proof}

\usepackage{wasysym}

\usepackage[left=2.7cm,right=2.7cm,top=2.7cm,bottom=2.7cm,bindingoffset=0cm]{geometry}

\title{On DP-coloring of graphs and multigraphs}
\date{}
\author{
Anton~Bernshteyn\thanks{Department of Mathematics, University of Illinois at Urbana--Champaign, IL, USA, \texttt{bernsht2@illinois.edu}. Research of this author is supported by the Illinois Distinguished Fellowship.}
\and Alexandr Kostochka\thanks{Department of Mathematics, University of Illinois at Urbana--Champaign, IL, USA and
Sobolev Institute of Mathematics, Novosibirsk 630090, Russia, \texttt{kostochk@math.uiuc.edu}. Research of this author is supported in part by NSF grant
 DMS-1266016 and by grant 15-01-05867 of the Russian Foundation for Basic Research.}
 \and Sergei Pron\thanks{Altai State University, Barnaul, Russia, \texttt{pspron@mail.ru}.}
 }

\makeatletter
\newcommand{\neutralize}[1]{\expandafter\let\csname c@#1\endcsname\count@}
\makeatother

\newtheorem{theo}{Theorem}

\newtheorem{lemma}[theo]{Lemma}

\newtheorem{corl}[theo]{Corollary}

\theoremstyle{definition}
\newtheorem{defn}[theo]{Definition}

\theoremstyle{remark}

	\def\quotient#1#2{%
		\raise1ex\hbox{$#1$}\Big/\lower1ex\hbox{$#2$}%
	}

	\renewcommand{\epsilon}{\varepsilon}

	\newcommand{\charge}[0]{\operatorname{ch}}

\begin{document}
	\maketitle
	
\begin{abstract}
While solving a question on list coloring of planar graphs, Dvo\v r\' ak and Postle introduced
the new notion of DP-coloring (they called it {\em correspondence coloring}). A~DP-coloring
of a graph $G$ reduces  the problem of finding a coloring of $G$ from a given 
 list $L$ to  the problem of finding a ``large'' independent set in an auxiliary graph $H(G,L)$ with
 vertex set $\{(v,c)\,: \, v\in V(G) \text{ and } {c\in L(v)} \}$. It is similar to the old reduction by Plesnevi\v{c} and Vizing
of the $k$-coloring problem to the problem of finding an independent
set of size $|V(G)|$ in the Cartesian product $G\square K_k$. Some properties of the DP-chromatic number $\chi_{DP}(G)$ resemble
the properties of the list chromatic number $\chi_{\ell}(G)$ but some  differ quite a lot.
It is always the case that  $\chi_{DP}(G)\geq \chi_{\ell}(G)$. The goal of this note is to introduce 
DP-colorings for multigraphs and to prove for them an analog of the result of Borodin and 
Erd\H{o}s--Rubin--Taylor characterizing the multigraphs that do not admit DP-colorings from
some DP-degree-lists. This characterization yields an analog of Gallai's Theorem on the minimum
number of edges in $n$-vertex graphs critical with respect to DP-coloring.
\end{abstract}

	\section{Introduction}	
 Graphs in this note  are assumed to be simple, i.e., they cannot have parallel edges or loops; multigraphs may have multiple edges but not loops.
 The complete $n$-vertex graph is denoted by $K_n$, and the $n$-vertex cycle is denoted by $C_n$.
  If $G$ is a (multi)graph and $v$,~$u \in V(G)$, then $E_G(v, u)$ denotes the set of all edges in $G$ connecting  $v$ and $u$,
 $e_G(v, u) \coloneqq |E_G(v, u)|$, and $\deg_G(v)\coloneqq\sum_{u\in V(G)\setminus \{v\}} e_G(v,u)$. For $A \subseteq V(G)$, $G[A]$ denotes the sub(multi)graph of $G$ induced by $A$, and for $A$, $B \subseteq V(G)$, $G[A,B]$ denotes the maximal bipartite sub(multi)graph of $G$ with parts $A$ and $B$. If $G_1$, \ldots, $G_k$ are graphs, then $G_1 + \ldots + G_k$ denotes the graph with vertex set $V(G_1)\cup \ldots \cup V(G_k)$ and edge set $E(G_1) \cup \ldots \cup E(G_k)$. The independence number of $G$ is denoted by $\alpha(G)$. For $k \in \mathbb{Z}_{> 0}$, $[k]$ denotes the set $\{1, \ldots, k\}$.

Recall that a (proper) $k$-{\em coloring} of a graph $G$ is a mapping $f\,:\,V(G)\to [k]$ such that $f(v)\neq f(u)$ whenever $vu\in E(G)$.
The smallest $k$ such that $G$ has a $k$-coloring is called the {\em chromatic number of} $G$ and is denoted by $\chi(G)$.
Plesnevi\v{c} and Vizing~\cite{PV} proved that $G$ has a $k$-coloring if and only if the Cartesian product $G\square K_k$ contains an independent set of size $|V(G)|$, i.e., $\alpha(G\square K_k)=|V(G)|$.

In order to tackle some graph coloring problems, Vizing~\cite{Vi} and independently   Erd\H{o}s,  Rubin, and  Taylor~\cite{ERT}
introduced the more general notion of {\em list coloring}. A {\em list} $L$ for a graph $G$ is a map $L : V(G) \to \operatorname{Pow}(\mathbb{Z}_{>0})$ that assigns to each vertex $v \in V(G)$ a set $L(v) \subseteq \mathbb{Z}_{> 0}$. An \emph{$L$-coloring} of $G$ is a mapping $f\,:\,V(G)\to \mathbb{Z}_{>0}$ such that 
$f(v)\in L(v)$ for each $v\in V(G)$ and  $f(v)\neq f(u)$ whenever $vu\in E(G)$. The {\em list chromatic number}, $\chi_{\ell}(G)$,
is the minimum $k$ such that $G$ has an $L$-coloring for each $L$ satisfying $|L(v)|=k$ for every $v\in V(G)$.

Since $G$ is $k$-colorable if and only if it is $L$-colorable with the list $L:v \mapsto [k]$, we have $\chi_{\ell}(G)\geq \chi(G)$ for every $G$; however, the difference $\chi_{\ell}(G)- \chi(G)$ can be arbitrarily large.
Moreover, graphs with chromatic number $2$ may have arbitrarily high list chromatic number. While $2$-colorable graphs may
have arbitrarily high minimum degree,  Alon~\cite{Al2000} showed that $\chi_{\ell}(G)\geq (1/2-o(1))\log_2 \delta$ for each graph $G$ with 
minimum degree $\delta$. On the other hand, some well-known upper bounds on $\chi(G)$ in terms of vertex degrees hold for
$\chi_{\ell}(G)$ as well. For example,  Brooks' theorem and the degeneracy upper bound hold for $\chi_{\ell}(G)$. 
Furthermore,  Borodin~\cite{B1,B2} and independently   Erd\H{o}s,  Rubin, and  Taylor~\cite{ERT} generalized Brooks' Theorem
to degree lists. Recall that a list $L$ for a graph $G$ is a {\em degree list} if $|L(v)|=\deg_G(v)$ for every $v\in V(G)$.

\begin{theo}[\cite{B1,B2,ERT}; a simple proof in~\cite{KSW}]\label{Bor}
	Suppose that $G$ is a connected graph. Then $G$ is not $L$-colorable for some degree list $L$ if and only if each block of $G$ is either a complete graph or an odd cycle.
\end{theo}

This result allows to extend Gallai's bound~\cite{G1} on the minimum number of edges in $n$-vertex $k$-critical graphs (i.e.,
graphs $G$ with $\chi(G)=k$ such that after deletion of any edge or vertex the chromatic number decreases) to $n$-vertex 
list-$k$-critical graphs (i.e.,
graphs $G$ with $\chi_{\ell}(G)=k$ such that after deletion of any edge or vertex the list chromatic number decreases).
	
List coloring proved useful in establishing a number of results for ordinary graph coloring; however, generally it is often much harder to prove
upper bounds on the list chromatic number than on chromatic number.  In order to prove such an upper bound for a class of planar 
 graphs, Dvo\v r\' ak and Postle~\cite{DP} introduced and heavily used a new generalization of list coloring;
they called it {\em correspondence coloring}, and we will call it {\em DP-coloring}, for short.

First, we show how to reduce to DP-coloring the problem of $L$-coloring of a graph~$G$. Given a list $L$ for $G$,
the vertex set of the auxiliary graph $H=H(G,L)$ is $\{(v,c)\,: \, v\in V(G) \text{ and } {c\in L(v)} \}$, and two distinct vertices $(v,c)$ 
and $(v',c')$ are adjacent in $H$ if and only if either $c=c'$ and $vv'\in E(G)$, or $v= v'$. Note that the independence number of $H$ is at most $|V(G)|$,
 since $V(H)$ is covered by $|V(G)|$ cliques. If $H$ 
has an independent set $I$ with $|I|=|V(G)|$, then, for each $v\in V(G)$, there is a unique $c\in L(v)$ such that $(v,c)\in I$. 
Moreover, the same color $c$ is not chosen for any two adjacent vertices. In other words, the map $f : V(G) \to \mathbb{Z}_{> 0}$ defined by $(v, f(v)) \in I$ is an $L$-coloring of $G$.
On the other hand, if $G$ has an $L$-coloring $f$, then the set $\{(v,f(v))\,:\,v\in V(G)\}$ is an independent set of size $|V(G)|$ in $H$.

By construction, for every distinct $v$, $v'\in V(G)$, the set of edges of $H$ connecting $\{(v,c)\,: \,  {c\in L(v)} \}$ and $\{(v',c')\,: \,  {c'\in L(v')} \}$
is empty if $vv'\notin E(G)$ and forms a matching (possibly empty) if $vv'\in E(G)$. Based on these properties of $H(G,L)$, 
Dvo\v r\' ak and Postle~\cite{DP} introduced the DP-coloring. The phrasing below is slightly different, but
the essence and the spirit are theirs. 	
	
	\begin{defn}\label{cov}
		Let $G$ be a graph. A \emph{cover} of $G$ is a pair $(L, H)$, where $L$ is an assignment of pairwise disjoint sets to the vertices of $G$ and $H$ is a graph with vertex set $\bigcup_{v \in V(G)} L(v)$, satisfying the following  conditions.
		\begin{enumerate}
			\item For each $v\in V(G)$, $H[L(v)]$ is a complete graph.
			\item For each $uv \in E(G)$, the edges between $L(u)$ and $L(v)$ form a matching (possibly empty).
			\item For each distinct $u$, $v\in V(G)$ with $uv \not\in E(G)$, no edges of $H$ connect   $L(u)$ and $L(v)$.
		\end{enumerate}
	\end{defn}
	
	\begin{defn}\label{defn:coloring}
Suppose $G$ is a graph and $(L, H)$ is a cover of $G$. An \emph{$(L,H)$-coloring} of $G$ is an independent set $I \subseteq V(H)$ of size $|V(G)|$. In this context, we refer to the vertices of $H$ as the \emph{colors}. $G$ is said to
 be \emph{$(L,H)$-colorable} if it admits an $(L,H)$-coloring.
	\end{defn}
	
	Note that if $(L, H)$ is a cover of $G$ and $I$ is an $(L, H)$-coloring, then $|I \cap L(v)| = 1$ for all $v \in V(G)$.
Fig.~\ref{C4} shows an example of two distinct covers of $G \cong C_4$.

\begin{figure}[ht]
\begin{center}
 \includegraphics[width=\textwidth]{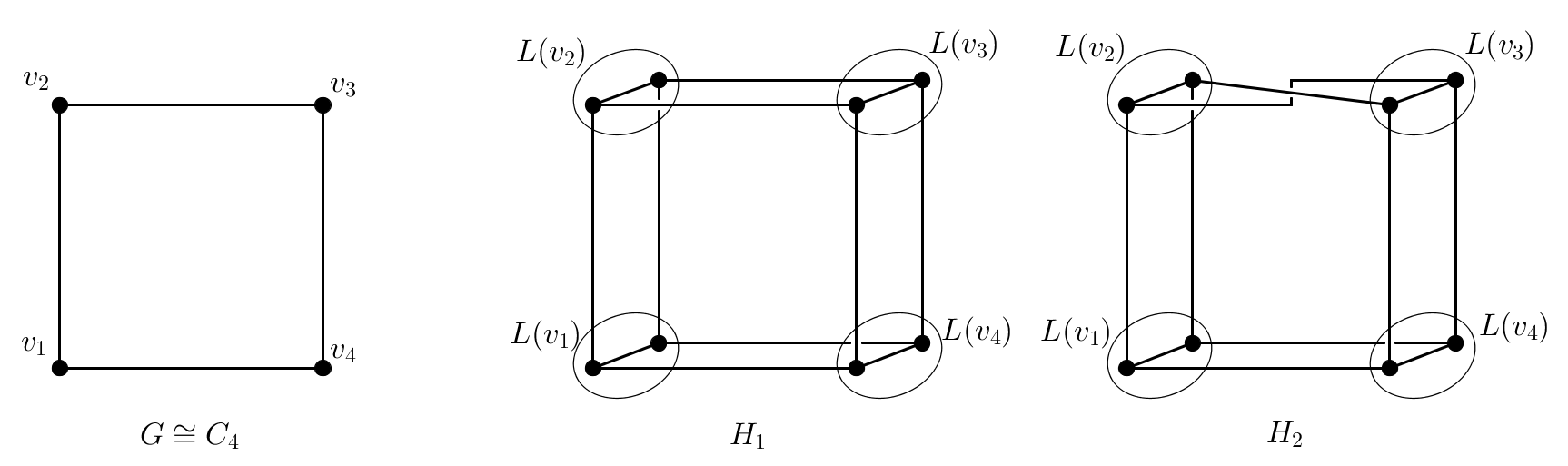}
\caption{\footnotesize Graph $C_4$ and two covers of it such that $C_4$ is $(L,H_1)$-colorable but not $(L, H_2)$-colorable.}\label{C4}
\end{center}
\end{figure}

	\begin{defn}\label{defn:fcolorable}
		Let $G$ be a graph and let $f \colon V(G) \to \mathbb{Z}_{\geq 0}$ be an assignment of nonnegative integers to the vertices of $G$.
 $G$ is \emph{DP-$f$-colorable} if it is $(L, H)$-colorable whenever $(L,H)$ is a cover of $G$ and $|L(v)| \geq f(v)$ for all $v \in V(G)$.
 If $G$ is DP-$\deg_G$-colorable, then $G$ is said to be \emph{DP-degree-colorable}.
	\end{defn}
	
\begin{defn}\label{defn:DPchr}
The {\em DP-chromatic number}, $\chi_{DP}(G)$, is the minimum $k$ such that $G$ is $(L,H)$-colorable for each choice of $(L,H)$
with $|L(v)| \geq k$ for all $v \in V(G)$.
	\end{defn}	

Dvo\v r\' ak and Postle observed that $\chi_{DP}(G)\leq k+1$ for every $k$-degenerate graph $G$ (we choose a color for each vertex $v$ greedily
from $L(v)$ avoiding the colors adjacent in $H$ to the colors already chosen for neighbors of $v$) and that Brooks' theorem almost holds for DP-colorings,
with the exception that $\chi_{DP}(C_n)=3$ for every cycle $C_n$ and not only for odd $n$, as for list coloring. The fact that $\chi_{DP}(C_4)=3$
marks an important difference between DP-coloring and list coloring since it implies that orientation theorems of Alon--Tarsi~\cite{AT} and the
Bondy--Boppana--Siegel Lemma (see~\cite{AT}) on list coloring do not extend to DP-coloring. Dvo\v r\' ak and Postle also mention that the proof
of Thomassen's theorem on list-$5$-colorability of planar graphs extends to DP-coloring. The first author of this note showed~\cite{Ber} that
the lower bound on DP-chromatic number of a graph $G$ with minimum degree $\delta$ is much stronger than Alon's bound~\cite{Al2000} for list coloring,
namely, that $\chi_{DP}(G)\geq \Omega(\delta/\ln \delta)$. On the other hand, he proved an analog of Johansson's upper bound~\cite{Jo}
on the list chromatic number of triangle-free graphs with given maximum degree.
	
The goal of this note is to naturally extend the notion of DP-coloring to multigraphs and to derive some simple properties of DP-colorings of
multigraphs. The main result is an analog of Theorem~\ref{Bor}: a characterization of connected multigraphs that are not DP-degree-colorable.
This result also yields a lower bound on the number of edges in $n$-vertex DP-critical graphs (we define such graphs in the next section).

The structure of the note is the following. In the next section we define DP-coloring of multigraphs and related notions, discuss some examples,
and state
our main result.  In Section 3 we prove the main result. In Section 4 we  briefly discuss DP-critical (multi)graphs 
and show a bound on the number of edges in them implied by the main result. For completeness, in the appendix we 
present Gallai's proof~\cite{G1} of his lemma on the number of edges in so-called Gallai trees (the original paper~\cite{G1} is in German).

\section{Definitions and the main result}

To define DP-coloring for multigraphs, we only need to change Definition~\ref{cov} as below and replace the word {\em graph} with
the word {\em multigraph} in Definitions~\ref{defn:coloring}--\ref{defn:DPchr}. The new version of Definition~\ref{cov} is:

	\begin{defn}\label{cov1}
		Let $G$ be a multigraph. A \emph{cover} of $G$ is a pair $(L, H)$, where $L$ is an assignment of pairwise disjoint sets to the vertices of $G$ and $H$ is a graph with vertex set $\bigcup_{v \in V(G)} L(v)$, satisfying the following  conditions.
		\begin{enumerate}
			\item For each $v\in V(G)$, $H[L(v)]$ is a complete graph.
			\item For each distinct $u$, $v\in V(G)$, the set of edges between $L(u)$ and $L(v)$ is the union of $e_G(u, v)$ (possibly empty) matchings.
		\end{enumerate}
	\end{defn}
		
	
	For a positive integer $k$ and a multigraph $G$, let $G^k$ denote the multigraph obtained from $G$ by replacing each edge in $G$ with a set of $k$ parallel edges.
In particular, $G^1=G$ for every $G$. The next two lemmas demonstrate two classes of multigraphs that are not DP-degree-colorable; 
the first of them exhibits multigraphs whose DP-chromatic number exceeds the number of vertices. In particular, for each $k\geq 2$, the $2$-vertex multigraph $K_2^k$
 has  DP-chromatic number $k+1$.	

	\begin{lemma}\label{lemma:completeisbad}
		The multigraph $K_n^k$ is not DP-degree-colorable.
	\end{lemma}
	\begin{proof}
		Let $G \coloneqq K_n^k$. For each $v \in V(G)$, let
		$$
			L(v) \coloneqq \{(v, i, j)\,:\, i \in [n-1], j \in [k]\},
		$$
		and let
		$$
			(v_1,i_1,j_1)(v_2, i_2, j_2) \in E(H) \,\vcentcolon\Longleftrightarrow\, v_1 = v_2\,\mbox{ or }\, i_1 =  i_2.
		$$
		Then $(L, H)$ is a cover of $G$ and $|L(v)| = k(n-1) = \deg_G(v)$ for all $v \in V(G)$. We claim that $G$ is not $(L,H)$-colorable. 
Indeed, if $I \subseteq V(H)$ is such that $|I \cap L(v)| = 1$ for all $v \in V$, then for some distinct $(v_1, i_1, j_1)$, $(v_2, i_2, j_2) \in I$, 
we have $i_1 = i_2$. Thus, $I$ is not an independent set. 
	\end{proof}
	
	\begin{lemma}\label{lemma:cycleisbad}
		The multigraph $C_n^k$ is not DP-degree-colorable.
	\end{lemma}
	\begin{proof}
		Let $G \coloneqq C_n^k$. Without loss of generality, assume that $V(G) = [n]$ and $e_G(u, v) = k$ if and only 
if $|u-v| = 1$ or $\{u,v\} = \{1,n\}$. For each $v \in [n]$, let
		$$
			L(v) \coloneqq \{(v, i, j)\,:\, i \in [2], j \in [k]\},
		$$
		and let
		$$
			(v_1, i_1, j_1)(v_2, i_2, j_2) \in E(H) \,\vcentcolon\Longleftrightarrow\, \left\{\begin{array}{l}
			v_1 = v_2; \text{ or} \\
			|v_1 - v_2|=1 \text{ and } i_1 = i_2; \text{ or}\\
			\{v_1, v_2\} = \{1, n\}\text{ and } i_1 = i_2 + 1 + n \pmod 2.
			\end{array}\right.
		$$
		Then $(L,H)$ is a cover of $G$ and $|L(v)| = 2k = \deg_G(v)$ for all $v \in [n]$. We claim that $G$ is not $(L, H)$-colorable. 
Indeed, suppose that $I \subset V(H)$ is an $(L,H)$-coloring of $G$. Let $I = \{(v, i_v, j_v)\}_{v = 1}^n$. 
Without loss of generality, assume that $i_1 = 1$. Then for each $v \in [n]$, $i_v = v \pmod 2$. 
Thus, $i_1 = i_n + 1 + n \pmod 2$, so $(1, i_1, j_1)(n, i_n,j_n) \in E(H)$; therefore, $I$ is not independent.
	\end{proof}
	
Our main result shows that the above lemmas describe all $2$-connected multigraphs that are not DP-degree-colorable.
	
	\begin{theo}\label{theo:main}
		Suppose that $G$ is a connected 
		multigraph. Then $G$ is not DP-degree-colorable if and only if each block of $G$ is  one of the graphs $K_n^k$, $C_n^k$ for some $n$ and $k$.
	\end{theo}

The result has an implication for the number of edges in DP-$k$-{\em critical} graphs and multigraphs, that is, (multi)graphs
$G$ with $\chi_{DP}(G)=k$ such that every proper sub(multi)graph of $G$ has a smaller DP-chromatic number. 
It is easy to show (and follows from the above lemmas and Lemma~\ref{lemma:toomany} in the next section) that $K_n^k$ is 
DP-$(k(n-1)+1)$-critical and $C_n^k$ is 
DP-$(2k+1)$-critical. It is also easy to show (and follows from Theorem~\ref{theo:main}) that
\begin{equation}\label{ner}
2|E(G)|\geq (k-1)n\quad\mbox{for every $n$-vertex DP-$k$-critical multigraph $G$.}
\end{equation}
The examples
of $C_n^k$ show that, for each odd $k\geq 3$, there are infinitely many $2$-connected  DP-$k$-critical multigraphs $G$ with
equality in~\eqref{ner}. However, if we consider only simple graphs, then Theorem~\ref{theo:main} implies a stronger bound than~\eqref{ner}, which is an analog of Gallai's bound~\cite{G1} for ordinary coloring (see~\cite{KSW} for list coloring):

\begin{corl}\label{cor1} Let $k\geq 4$ and let $G$ be a DP-$k$-critical graph distinct from $K_k$. Then
\begin{equation}\label{ner2}
2|E(G)|\geq \left(k-1+\frac{k-3}{k^2-3}\right)n.
\end{equation}
\end{corl}

We will prove Theorem~\ref{theo:main} in the next section and derive Corollary~\ref{cor1} in Section~4.

\section{Proof of Theorem~\ref{theo:main}}
	We proceed via a series of lemmas.
		
	\begin{lemma}\label{lemma:cycle}
		Suppose that $G$ is a regular $n$-vertex multigraph whose underlying simple graph is a cycle. 
Then $G$ is not DP-degree-colorable if and only if $G \cong C_n^k$ for some $k$.
	\end{lemma}
	\begin{proof}
		Without loss of generality, assume that $V(G) = [n]$ and $e_G(u, v) > 0$ if and only if $|u-v| = 1$ or $\{u,v\} = \{1,n\}$. 
Suppose that $G \not \cong C_n^k$. Since $G$ is regular, this implies that $n$ is even and for some distinct positive $r$, $s$, $e_G(v, v+1) = r$ 
for all odd $v \in [n]$ and $e_G(1, n) = e_G(v, v+1) = s$ for all even $v \in [n-1]$. Without loss of generality, assume $s > r$.
		
		Let $(L, H)$ be a cover of $G$ such that $|L(v)| = \deg_G(v) = r+s$ for all $v \in [n]$.
 We will show that $G$ is $(L, H)$-colorable. For $x \in L(1)$, say that a color $y \in L(v)$ is \emph{$x$-admissible} 
if there exists a set $I \subseteq V(H)$ that is independent in $H - E_H(L(1), L(n))$ such that $|I \cap L(u)| = 1$ for all $u \in [v]$ 
and $\{x, y\} \subseteq I$. Let $A_x(v)\subseteq L(v)$ denote the set of all $x$-admissible colors in $L(v)$. 
Clearly, for each $x \in L(1)$, $|A_x(2)| \geq s$ and $|A_x(3)| \geq r$. Suppose that for some $x \in L(1)$, $|A_x(3)| > r$. 
Since each color in $L(4)$ has at most $r$ neighbors in $L(3)$, $A_x(4) = L(4)$. 
Similarly, $A_x(v) = L(v)$ for all $v \geq 4$. In particular, $A_x(n) = L(n)$. Take any $y \in L(n) \setminus N_H(x)$. Since $y \in A_x(n)$, there exists a 
set $I \subseteq V(H)$ that is independent in $H - E_H(L(1), L(n))$ such that $|I \cap L(u)| = 1$ for all $u \in [n]$ and $\{x, y\} \subseteq I$. 
But then $I$ is independent in $H$, so $I$ is an $(L, H)$-coloring of $G$. Thus, we may assume that $|A_x(3)| = r$ for all $x \in L(v)$. Note that
		$$
			L(3) \setminus A_x(3) = L(3) \cap \bigcap_{y \in A_x(2)} N_H(y).
		$$
		Therefore, $L(3) \cap N_H(y)$ is the same set of size $s$ for all $y \in A_x(2)$. 
Since each vertex in $L(3)$ has at most $s$ neighbors in $L(2)$, the graph $H[A_x(2)\cup (L(3) \setminus A_x(3))]$ is a complete $2s$-vertex graph. Since every vertex in $L(2)$ is $x$-admissible for some $x \in L(1)$, 
$H[L(2)\cup L(3)]$ contains a disjoint union of at least two complete $2s$-vertex graphs. Therefore, $|L(2)\cup L(3)| \geq 4s$. But $|L(2)| = |L(3)|=r + s < 2s$; a contradiction.
	\end{proof}
	
	\begin{lemma}\label{lemma:toomany}
		Let $G$ be a connected multigraph and suppose $(L, H)$ is a cover of $G$ such that $|L(v)| \geq \deg_G(v)$ 
for all $v \in V(G)$, and  $|L(v_0)| > \deg_G(v_0)$ for some $v_0 \in V(G)$. Then $G$ is $(L, H)$-colorable.
	\end{lemma}
	\begin{proof}
		If $|V(G)| = 1$, the statement is clear. Now suppose $G$ is a counterexample with the 
fewest vertices. 
Consider the multigraph $G' \coloneqq G - v_0$. For each $v \in V(G')$, let $L'(v) \coloneqq L(v)$, and let $H' \coloneqq H - L(v_0)$. By construction, $(L', H')$ is a cover of $G'$ such that for all $v \in V(G')$, $|L(v)| \geq \deg_{G'}(v)$. Moreover, since $G$ is connected,
 each connected component of $G'$ contains a vertex $u$ adjacent in $G$ to $v_0$ and thus satisfying  $\deg_{G'}(u) < \deg_G(u)$. 
Hence, by the minimality assumption, $G'$ is $(L', H')$-colorable. Let $I' \subseteq V(H')$ be an $(L', H')$-coloring of $G'$. 
Then $|N_G(I') \cap L(v_0)| \leq \deg_G(v_0)$, so $L(v_0) \setminus N_G(I') \neq \emptyset$. 
Thus, $I'$ can be extended to an $(L, H)$-coloring $I$ of $G$; a contradiction.
	\end{proof}
	
	\begin{lemma}\label{lemma:saturated}
		Let $G$ be a connected multigraph and let $(L, H)$ be a cover of $G$. Suppose that there is a vertex $v_1 \in V(G)$ and a color 
$x_1 \in L(v_1)$ such that $G - v_1$ is connected and for some $v_2 \in V(G)\setminus \{v_1\}$, $x_1$ has fewer than $e_G(v_1, v_2)$ neighbors in $L(v_2)$. 
Then $G$ is $(L, H)$-colorable.
	\end{lemma}
	\begin{proof}
		Let $G' \coloneqq G - v_1$. For each $v \in V(G')$, let
		$$
			L'(v) \coloneqq L(v) \setminus N_H(x_1), 
		$$ 		and let		$$
			H' \coloneqq H - L(v_1) - N_H(x_1).
		$$
		Then $(L', H')$ is a cover of $G'$. Moreover, for each $v \in V(G')$,
		$$
			|L'(v)| = |L(v)| - |L(v) \cap N_H(x_1)| \geq \deg_G(v) - e_G(v, v_1) = \deg_{G'}(v),
		$$
		and
		$$
			|L'(v_2)| = |L(v_2)| - |L(v_2) \cap N_H(x_1)| > \deg_G(v_2) - e_G(v_2, v_1) = \deg_{G'}(v_2).
		$$		
		Since $G'$ is connected, Lemma~\ref{lemma:toomany} implies that $G'$ is $(L', H')$-colorable. 
But if $I' \subseteq V(H')$ is an $(L',H')$-coloring of $G'$, then $I' \cup \{x_1\}$ is an $(L,H)$-coloring of $G$, as desired.
	\end{proof}
	
	\begin{lemma}\label{lemma:regular}
		Suppose that $G$ is a $2$-connected multigraph and $(L,H)$ is a cover of $G$ with $|L(v)|\geq \deg_G(v)$ for each $v\in V(G)$. If
$G$ is not $(L,H)$-colorable, then $G$ is regular and
for each pair of adjacent vertices $v_1$, $v_2\in V(G)$, the bipartite graph $H[L(v_1), L(v_2)]$ is $e_G(v_1, v_2)$-regular.
	\end{lemma}
	\begin{proof}
Consider any two adjacent $v_1$, $v_2 \in V(G)$. By Lemma~\ref{lemma:saturated}, $H[L(v_1), L(v_2)]$ is a $e_G(v_1, v_2)$-regular bipartite graph with parts $L(v_1)$, $L(v_2)$. Therefore, $|L(v_1)| = |L(v_2)|$, so $\deg_G(v_1) = \deg_G(v_2)$, as desired. Since $G$ is connected and $v_1$, $v_2$ are arbitrary adjacent
vertices in $G$, this yields $G$ is regular.
	\end{proof}
	
	\begin{lemma}\label{lemma:triangle}
		Let $G$ be a $2$-connected multigraph. Suppose that $u_1$, $u_2$, $w \in V(G)$ are distinct vertices such that $G - u_1 - u_2$ is connected,
 $e_G(u_1, u_2) < e_G(u_1, w)$, and $e_G(u_2, w) \geq 1$. Then $G$ is DP-degree-colorable.
	\end{lemma}
	\begin{proof}
	Suppose $G$ is not $(L,H)$-colorable for some cover $(L, H)$ with $|L(v)| = \deg_G(v)$ for all $v \in V(G)$. 
First we show that
\begin{equation}\label{0223}
 \mbox{there are nonadjacent $x_1 \in L(u_1)$, $x_2 \in L(u_2)$ with $N_H(x_1) \cap N_H(x_2) \cap L(w) \neq \emptyset$.}
\end{equation}
Indeed, consider any $x_2 \in L(u_2)$. By Lemma~\ref{lemma:regular}, $|L(w) \cap N_H(x_2)| = e_G(u_2, w) \geq 1$. 
Similarly, for each $y \in L(w) \cap N_H(x_2)$, $|L(u_1) \cap N_H(y)| = e_G(u_1, w) > e_G(u_1, u_2) = |L(u_1) \cap N_H(x_2)|$.
 Thus, there exists $x_1 \in (L(u_1) \cap N_H(y))\setminus (L(u_1) \cap N_H(x_2))$. By the choice, $x_1$ and $x_2$ are nonadjacent and
 $y \in N_H(x_1) \cap N_H(x_2) \cap L(w)$. This proves~\eqref{0223}.

Let $x_1$ and $x_2$ satisfy~\eqref{0223}. Let $G' \coloneqq G - u_1-u_2$. For each $v \in V(G')$, let
		$$
			L'(v) \coloneqq L(v) \setminus (N_H(x_1)\cup N_H(x_2)),
		$$
		and let
		$$
			H' \coloneqq H - L(u_1)-L(u_2) - N_H(x_1)-N_H(x_2).
		$$
		Then $G'$ is connected and $(L', H')$ is a cover of $G'$ satisfying the conditions of Lemma~\ref{lemma:toomany} with $w$ in the role
of $v_0$. Thus $G'$ is $(L', H')$-colorable, and hence $G$ is $(L, H)$-colorable, a contradiction.
	\end{proof}
	
	\begin{lemma}\label{lemma:politician}
		Suppose that $G$ is an $n$-vertex $2$-connected multigraph that contains a vertex adjacent to all other vertices. 
Then either $G \cong K_n^k$ for some $k$, or $G$ is DP-degree-colorable.
	\end{lemma}
	\begin{proof}
		Suppose that $G$ is an $n$-vertex multigraph that is not DP-degree-colorable and assume that $w \in V(G)$ is adjacent to all other vertices.
 If some distinct $u_1$, $u_2 \in V(G)\setminus \{w\}$ are nonadjacent, then the triple $u_1$, $u_2$, $w$ satisfies the conditions of 
Lemma~\ref{lemma:triangle}, so $G$ is DP-degree-colorable. Hence any two vertices in $G$ are adjacent; in other words, 
the underlying simple graph of $G$ is  $K_n$. It remains to show that any two vertices in $G$ are connected by the same number of edges.
 Indeed, if $u_1$, $u_2$, $u_3 \in V(G)$ are such that $e_G(u_1, u_2) < e_G(u_1, u_3)$, then, by Lemma~\ref{lemma:triangle} again, $G$ is DP-degree-colorable.
	\end{proof}
	
	\begin{lemma}\label{lemma:degree2}
		Suppose that $G$ is a $2$-connected $n$-vertex multigraph in which each vertex has at most $2$ neighbors. 
Then either $G \cong C_n^k$  for some $k$, or $G$ is DP-degree-colorable.
	\end{lemma}
	\begin{proof}
		Suppose that $G$ is a $2$-connected $n$-vertex multigraph in which each vertex has at most $2$ neighbors 
and that is not DP-degree-colorable. Then the underlying simple graph of $G$ is a cycle and Lemma~\ref{lemma:regular}
 implies that $G$ is regular, so $G \cong C_n^k$ by Lemma~\ref{lemma:cycle}.
	\end{proof}
	
	\begin{lemma}\label{lemma:2connected}
		Suppose that $G$ is a $2$-connected $n$-vertex multigraph that is not DP-degree-colorable. Then $G \cong K_n^k$ or $C_n^k$ for some $k$.
	\end{lemma}
	\begin{proof}
		By Lemmas~\ref{lemma:politician} and \ref{lemma:degree2}, we may assume that $G$ contains a vertex $u$ 
such that $3 \leq |N_G(u)| \leq n-2$. Since $G$ is $2$-connected, $G - u$ is connected. However, $G-u$ is not $2$-connected. 
Indeed, let $v_1$ be any vertex in $V(G)\setminus (\{u\} \cup N_G(u))$ that shares a neighbor $w$ with $u$. 
Due to Lemma~\ref{lemma:triangle} with $u$ in place of $v_2$, $G - v_1 - u$ is disconnected, so $v_1$ is a cut vertex in $G - u$.
		
		Therefore, $G - u$ contains at least two leaf  blocks, say $B_1$ and $B_2$. For $i\in[2]$, let $x_i$ be the cut
vertex of $G-u$ contained in $B_i$.
Since $G$ itself is $2$-connected, $u$ has a neighbor $v_i\in B_i-x_i$ for each $i\in[2]$.
Then $v_1$ and $v_2$ are nonadjacent and $G - u - v_1 - v_2$ is connected. Since $u$ has at least $3$ neighbors, 
$G - v_1 - v_2$ is also connected. Hence, we are done by Lemma~\ref{lemma:triangle} with $u$ in the role of $w$. 
	\end{proof}
	
	\begin{lemma}\label{lemma:glueing}
		Suppose that $w\in V(G)$, $G = G_1 + G_2$, and $V(G_1)\cap V(G_2)=\{w\}$. 
 If  $G_1$ and $G_2$ are not DP-degree-colorable, then $G$ is not DP-degree-colorable.
	\end{lemma}
	\begin{proof}
 Suppose that $G_1$ is not $(L_1, H_1)$-colorable and $G_2$ is not $(L_2, G_2)$-colorable, 
where for each $i \in [2]$, $(L_i, H_i)$ is a cover of $G_i$ such that $|L(v)| = \deg_{G_i}(v)$ for all $v \in V(G_i)$. 
Without loss of generality, assume that $L_1(v_1) \cap L_2(v_2) = \emptyset$ for all $v_1 \in V(G_1)$, $v_2 \in V(G_2)$. For each $v \in V(G)$, let
		$$
			L(v) \coloneqq \begin{cases}
				L_1(v) &\text{if $v \in V(G_1) \setminus \{w\}$};\\
				L_2(v) &\text{if $v \in V(G_2) \setminus \{w\}$};\\
				L_1(w) \cup L_2(w) &\text{if $v = w$},
			\end{cases}
		$$
		and let $H \coloneqq H_1 + H_2 + K(L(w))$, where $K(L(w))$ denotes the complete graph with vertex set $L(w)$. Then $(L, H)$ is a cover of $G$ and for each $v \in V(G)$, $|L(v)| = \deg_G(v)$. 
Suppose that $G$ is $(L, H)$-colorable and let $I$ be an $(L, H)$-coloring of $G$. Without loss of generality,
 assume $I \cap L(w) \subseteq L_1(w)$. Then $I \cap V(H_1)$ is an $(L_1, H_1)$-coloring of $G_1$; a contradiction.
	\end{proof}
	
	\begin{proof}[Proof of Theorem~\ref{theo:main}]
		Lemmas~\ref{lemma:completeisbad}, \ref{lemma:cycleisbad}, and \ref{lemma:glueing} show that if each block 
of $G$ is isomorphic to one of the multigraphs $K_n^k$, $C_n^k$ for some $n$ and $k$, then $G$ is not DP-degree-colorable. 

Now assume that $G$ is a connected multigraph that is not DP-degree-colorable. If $G$ is $2$-connected, then we are done by 
Lemma~\ref{lemma:2connected}. Therefore, we may assume that $G$ has a cut vertex $w \in V(G)$. Let $G_1$ and $G_2$ be nontrivial 
connected subgraphs of $G$ such that $G = G_1 + G_2$ and $V(G_1) \cap V(G_2) = \{w\}$. It remains to show that neither $G_1$ nor $G_2$ 
is DP-degree-colorable, since then we will be done by induction. Suppose towards a contradiction that $G_1$ is DP-degree-colorable. 
Let $(L, H)$ be a cover of $G$ such that $|L(v)| =\deg_G(v)$ for all $v \in V(G)$. Due to Lemma~\ref{lemma:toomany} applied to the 
connected components of $G_2 - w$, there exists an independent set $I_2 \subseteq \bigcup_{v \in V(G_2) \setminus \{w\}} L(v)$ such that
 $|L(v) \cap I_2| = 1$ for all $v \in V(G_2) \setminus \{w\}$. For each $v \in V(G_1)$, let
		$$
			L_1(v) \coloneqq L(v) \setminus N_H(I_2).
		$$
		(Note that $L_1(v) = L(v)$ for all $v \in V(G_1) \setminus \{w\}$.) Also, let
		$$
			H_1\coloneqq H \left[\bigcup_{v \in V(G_1)} L_1(v)\right].
		$$
		Then $(L_1, H_1)$ is a cover of $G_1$. For each $v \in V(G_1)\setminus\{w\}$, $|L_1(v)| = |L(v)| = \deg_G(v) = \deg_{G_1}(v)$; 
and for $w$ we have $|L_1(w)| = |L(v)| - |N_H(I_2) \cap L(w)| \geq \deg_G(w) - \deg_{G_2}(w) = \deg_{G_1}(w)$. 
Since $G_1$ is DP-degree-colorable, it is $(L_1, H_1)$-colorable. But if $I_1$ is an $(L_1, H_1)$-coloring of $G_1$, then $I_1 \cup I_2$ is an $(L, H)$-coloring of $G$.
	\end{proof}
\section{On DP-critical graphs}	

Gallai~\cite{G1} proved bound~\eqref{ner2} for ordinary $k$-critical $n$-vertex graphs using an upper bound on the number of edges in
{\em Gallai trees}---the graphs in which every block is a complete graph or an odd cycle. We will need the same statement for
{\em GDP-trees}---the graphs in which every block is a complete graph or a cycle (not necessarily odd).

\begin{lemma}\label{lem1} Let $k\geq 4$ and let $T$ be an $n$-vertex GDP-tree  with maximum degree $\Delta(T)\leq k-1$ not
containing $K_k$. Then
\begin{equation}\label{e1}
2|E(T)|\leq \left(k-2+\frac{2}{k-1}\right)n.
\end{equation}
\end{lemma}
The proof is the same as Gallai's. We present the proof in the appendix, since Gallai's paper is in German.
Below is the rest of the proof of Corollary~\ref{cor1}. It is based on Gallai's ideas but is shorter.

We use discharging. Let $G$ be an $n$-vertex DP-$k$-critical graph distinct from $K_k$. Note that the minimum degree of $G$ is at least $k-1$. 
The initial charge of each vertex $v\in V(G)$ is $\charge(v)\coloneqq\deg_G(v)$. The only discharging rule is this:

{\bf (R1)} Each vertex $v\in V(G)$ with $\deg_G(v)\geq k$ sends to each neighbor the charge $\frac{k-1}{k^2-3}$.

Denote the new charge of each vertex $v$ by $\charge^\ast(v)$. We will show that 
\begin{equation}\label{ch}
 \sum_{v\in V(G)}\charge^\ast(v)\geq \left(k-1+\frac{k-3}{k^2-3}\right)n.
\end{equation}

Indeed, if $\deg_G(v)\geq k$, then 
\begin{equation}\label{ch*}
\charge^\ast(v)\geq \deg_G(v)-\frac{k-1}{k^2-3}\cdot\deg_G(v)\geq k\left(1-\frac{k-1}{k^2-3}\right)=k-1+\frac{k-3}{k^2-3}.
\end{equation}
Also, if $T$ is a component of the subgraph $G'$ of $G$ induced by the vertices of degree $k-1$, then
$$ \sum_{v\in V(T)}\charge^\ast(v)\geq(k-1)|V(T)|+\frac{k-1}{k^2-3}\left|E_G(V(T),V(G)\setminus V(T)\right|.$$
Since $T$ is a GDP-tree and does not contain $K_k$, by Lemma~\ref{lem1},
$$\left|E(V(T),V(G)\setminus V(T)\right|\geq (k-1)|V(T)| - \left(k-2+\frac{2}{k-1}\right)|V(T)| = \frac{k-3}{k-1}|V(T)|.$$ Thus for every component $T$ of $G'$ we have
$$\sum_{v\in V(T)}\charge^\ast(v)\geq (k-1)|V(T)|+\frac{k-1}{k^2-3}\cdot\frac{k-3}{k-1}\cdot|V(T)|=\left(k-1+\frac{k-3}{k^2-3}\right)|V(T)|.$$
Together with~\eqref{ch*}, this implies~\eqref{ch}.

\section*{Appendix}
We repeat Gallai's proof of Lemma~\ref{lem1} by induction on the number of blocks. If $T$ is a block, then, since $T\not\cong K_k$ and $k\geq 4$,
$\Delta(T)\leq k-2$ which is stronger than~\eqref{e1}.

Suppose~\eqref{e1} holds for all GDP-trees with at most $s$ blocks and $T$ is a GDP-tree with $s+1$ blocks.
Let $B$ be a leaf block in $T$ and $x$ be the cut vertex in $V(B)$. Let $D\coloneqq\Delta(B)$.

{\bf Case 1:}  $D\leq k-3$. Let $T'\coloneqq T-(V(B)\setminus \{x\})$. Then $T'$ is a GDP-tree with $s$ blocks.
So $2|E(T)|=2|E(T')|+D|V(B)|$ and, by induction, $2|E(T')|\leq \left(k-2+\frac{2}{k-1}\right)(n-|V(B)|+1)$. If $B=K_r$, then $r=D+1\leq k-2$. So in this case
\begin{align*}
	&2|E(T)|-\left(k-2+\frac{2}{k-1}\right)n\\
	\leq\,&\left(k-2+\frac{2}{k-1}\right)(n-D)+D(D+1)-\left(k-2+\frac{2}{k-1}\right)n\\
	=\,&D\left(-k+2-\frac{2}{k-1} + D + 1\right)\leq -D\frac{2}{k-1}<0,
\end{align*}
as claimed. Similarly, if $B=C_t$, then, by the case, $k\geq 5$ and
\begin{align*}
	&2|E(T)|-\left(k-2+\frac{2}{k-1}\right)n\\
	\leq\,&\left(k-2+\frac{2}{k-1}\right)(n-t+1)+2t-n\left(k-2+\frac{2}{k-1}\right)\\
	=\,&(t-1)\left(-k+2-\frac{2}{k-1}+2\right)+2< 2\left(-k+4\right)+2\leq 0.
\end{align*}

{\bf Case 2:}  $D= k-2$. Since $\Delta(T)\leq k-1$, only one block $B'$ apart from $B$ may contain $x$
and this $B'$ must be $K_2$. 
Let $T''=T-V(B)$.  Then $T''$ is a GDP-tree with $s-1$ blocks.
So $2|E(T)|=2|E(T'')|+D|V(B)|+2$ and, by induction, $2|(T'')|\leq \left(k-2+\frac{2}{k-1}\right)(n - |V(B)|)$. Hence in this case, since $|V(B)|\geq D+1=k-1$,
\begin{align*}
	&2|E(T)|-\left(k-2+\frac{2}{k-1}\right)n\\
	\leq\,&\left(k-2+\frac{2}{k-1}\right)(n-|V(B)|)+(k-2)|V(B)|+2-\left(k-2+\frac{2}{k-1}\right)n\\
	=\,&|V(B)|\left(-k+2-\frac{2}{k-1}+k-2\right)+2\leq -\frac{2}{k-1}|V(B)|+2\leq 0,
\end{align*}
again.\qed
\end{document}